\documentclass[reqno]{amsart}
\pdfoutput=1
\usepackage{amssymb}
\usepackage{amsthm}
\usepackage{bbm}
\usepackage{graphicx}

\usepackage[colorlinks=true]{hyperref}
\usepackage[all]{xypic}

\usepackage[final]{epsfig}
\usepackage{psfrag}
\usepackage{epstopdf}

\newtheorem{thm}{Theorem}[section]

\newtheorem{proposition}[thm]{Proposition}
\newtheorem{corollary}[thm]{Corollary}

\newtheorem{cor}[thm]{Corollary}
\newtheorem{lem}[thm]{Lemma}
\newtheorem{prop}[thm]{Proposition}

\theoremstyle{definition}
\newtheorem{definition}{Definition}[section]

\theoremstyle{observation}

\theoremstyle{definition}
\newtheorem{example}{Example}[section]

\newcommand{\imp}{\frak{I}}
\newcommand{\pimp}{\frak{P}}

\newcommand{\B}{\mathcal{B}}

\newcommand{\M}{\mathcal{M}}

\newcommand{\ccl}[1]{{\mathbf {#1}}} 

\newcommand{\uc}{\mathbb{S}^1}

\newcommand{\defin}[1]{{\it #1}}
\newcommand{\R}{\mathbb{R}}
\newcommand{\N}{\mathbb{N}}
\newcommand{\Q}{\mathbb{Q}}

\theoremstyle{theorem}

{\bf}{\it}
{\bf}{\it}
\newtheorem*{riemann-mapping-theorem}{Riemann Mapping  Theorem}
{\bf}{\it}
{\bf}{\it}
{\bf}{\it}
{\bf}{\it}
{\bf}{\it}
\newtheorem*{fshb}{Fatou-Shishikura Bound}{\bf}{\it}
\newtheorem{tet}{Conditional Implication}{\bf}{\it}
\newtheorem*{carat}{Carath{\'e}odory Theorem}{\bf}{\it}
\newtheorem*{baby-carat}{Carath{\'e}odory Theorem for locally connected domains}{\bf}{\it}
{\bf}{\it}
{\bf}{\it}

\theoremstyle{remark}
\newtheorem{remark}[thm]{Remark}

\newenvironment{pf*}[1]{\proof[#1]}{\endproof}
\usepackage{euscript}

\newcommand{\cal}[1]{{\mathcal #1}}

\newcommand{\nothing}[1]{{}}

\newcommand{\beq}{\begin{equation}}
\newcommand{\eeq}{\end{equation}}

\newtheorem{defn}{Definition}[section]

\newtheorem{rem}{Remark}[section]

\renewcommand{\deg}{\operatorname{deg}}
\newcommand{\riem}{{\widehat{\CC}}}
\newcommand{\diam}{\operatorname{diam}}

\newcommand{\dist}{\operatorname{dist}}

\newcommand{\eps}{\epsilon}

\numberwithin{equation}{section}
\newcommand{\thmref}[1]{Theorem~\ref{#1}}
\newcommand{\propref}[1]{Proposition~\ref{#1}}

\newcommand{\corref}[1]{Corollary~\ref{#1}}

\newcommand{\supp}{\operatorname{Supp}}

\newcommand{\cI}{{\cal I}}
\newcommand{\cA}{{\cal A}}

\newcommand{\cM}{{\cal M}}

\newcommand{\cB}{{\cal B}}

\renewcommand{\cD}{{\cal D}}

\newcommand{\cS}{{\cal S}}

\newcommand{\cK}{{\cal K}}

\newcommand{\CC}{{\mathbb C}}
\newcommand{\RR}{{\mathbb R}}

\newcommand{\ZZ}{{\mathbb Z}}
\newcommand{\NN}{{\mathbb N}}
\newcommand{\DD}{{\mathbb D}}

\newcommand{\QQ}{{\mathbb Q}}

\newcommand{\ov}[1]{\overline{#1}}

\newcommand{\ignore}[1]{{}}

\title[Computable geometric analysis and  dynamics]{Computable geometric complex analysis and complex dynamics}

\author{Cristobal Rojas}
\address{Departamento de Matem\'aticas, Universidad Andres Bello\\  Rep\'ublica 498, Santiago, Chile.}
\thanks{C.R. was partially supported by Proyecto FONDECYT  No. 1150222 and Basal PFB-03 CMM - Universidad de Chile. }

\author{Michael Yampolsky}
\address{Department of Mathematics, University of Toronto\\  40 St. George street, Toronto, ON, Canada.}
\thanks{M.Y. was partially supported by NSERC Discovery Grant.} 

\begin{document}
\begin{abstract}
We discuss computability and computational complexity of conformal mappings  and their boundary extensions. As applications, we review the state of the art regarding computability and complexity of Julia sets, their invariant measures and external rays impressions.  
\end{abstract}

\maketitle

\section{Introduction}
The purpose of this paper is to survey  the exciting recent applications of Computable Analysis to Geometric Complex Analysis and
Complex Dynamics. 
Computable Analysis was founded by Banach and Mazur \cite{BM} in 1937,
only one year after the birth of Turing Machines and Post Machines.
Interrupted by war, this work was further developed in the book
by Mazur \cite{Maz}, and followed in mid-1950's by the works
of Grzegorczyk \cite{Grz},  Lacombe \cite{Lac}, and others. A parallel school of 
Constructive Analysis was founded by  A.~A.~Markov in Russia in the late 1940's.
A modern treatment of the field can be found in \cite{Ko} and \cite{Wei}.

In the last two decades, there has been much activity both by theorists and by computational practitioners in applying 
natural notions of computational hardness to such complex-analytic objects as the Riemann Mapping, Carath{\'e}odory Extension,
the Mandelbrot set, and Julia sets. The goal of this survey is to  give a brief summary of the beautiful interplay between Computability,
Analysis, and Geometry, which
has emerged from these works. 
\section{Required computability notions}

In this section we gather all the computability notions involved in the statements presented in the survey; the reader is referred to  \cite{Wei} for the details.

\subsection{Computable metric spaces}

\begin{definition}
A \defin{computable metric space} is a triple $(X,d,\cS)$ where:
\begin{enumerate}
\item $(X,d)$ is a separable metric space,
\item $\cS=\{s_i:i\in\N\}$ is a dense sequence of points in $X$,
\item $d(s_i,s_j)$ are computable real numbers, uniformly in $(i,j)$.
\end{enumerate}
\end{definition}

The  points in $\cS$ are called  \defin{ideal}. 
\begin{example}
A basic example is to take the space $X=\RR^n$ with the usual notion of Euclidean distance $d(\cdot,\cdot)$, and 
to let the set $\cS$ consist of points $\bar x=(x_1,\ldots,x_n)$ with rational coordinates. In what follows,
we will implicitly make these choices of $\cS$ and $d(\cdot,\cdot)$ when discussing computability in $\RR^n$.
\end{example}

\begin{definition}A point $x$ is \defin{computable} if there is a computable function
$f:\NN\to\cS$  such that $$d(f(n),x)<2^{-n}\text{ for all }n.$$
\end{definition}

If $x\in X$ and $r>0$, the metric ball $B(x,r)$ is defined as 
$$B(x,r)=\{y\in X:d(x,y)<r\}.$$
Since the set $\cB:=\{B(s,q):s\in\cS,q\in\Q, q>0\}$ of \defin{ideal balls} is countable, we can fix an enumeration
$\B=\{B_i:i\in\N\}$, which we assume to be effective with respect to the enumerations of $\cS$ and $\QQ$.

\begin{definition}
 An open set $U$ is called \defin{lower-computable} if there is a computable function $f:\NN\to\NN$
 such that $$U=\bigcup_{n\in \NN}B_{f(n)}.$$ 
\end{definition}
It is not difficult to see that finite intersections or infinite unions of (uniformly) lower-computable open sets are again lower computable.

Having defined lower-computable open sets, we naturally proceed to the following definitions for closed sets.
\begin{definition}\label{d.upper.comp} A closed set $K$  is called  \defin{upper-computable} if its complement is a lower-computable open set, and   \defin{lower-computable} if the relation $K\cap B_{i}\neq \emptyset$ is semi-decidable, uniformly in $i$.
\end{definition}


\noindent
In other words, a closed set $K$ is lower-computable if there exists an algorithm $\cA$ which enumerates all ideal balls which have  non-empty intersection with $K$.
To see that this definition is a natural extension of lower computability of open sets, we note:
\begin{example} The closure $\ov{U}$ of any open lower-computable set $U$ is lower-computable since $B_{i}\cap \ov{U}\neq \emptyset$ if and only if there exists $s\in  B_{i}\cap U$, which is uniformly semi-decidable for such $U$.



\end{example}

The following is  a useful characterization of lower-computable sets:
\begin{proposition}\label{p.comp-closed}
A closed set $K$ is lower-computable if and only if  there exists a sequence of uniformly computable  points $x_{i}\in K$ which is dense in $K$.
 \end{proposition}
\begin{proof}
Observe that, given some ideal ball $B=B(s,q)$ intersecting $K$, the relations $\ov{B_i}\subset B$, $ q_{i}< 2^{-k}$ and $B_{i}\cap K\neq \emptyset$ are all semi-decidable and then we can find an exponentially decreasing sequence of ideal balls $(B_{k})$ intersecting $K$. Hence $\{x\}=\cap_k B_k$ is a computable point lying in $B\cap K$.

The other direction is obvious.
\end{proof}

\begin{definition}A closed set is \defin{computable} if it is lower and upper computable.
\end{definition}

Here is an alternative way to define a computable set. Recall that {\it Hausdorff distance} between two nonempty
compact sets $K_1$, $K_2$ is
$$\dist_H(K_1,K_2)=\inf_\eps\{K_1\subset U_\eps(K_2)\text{ and }K_2\subset U_\eps(K_1)\},$$
where $U_\eps(K)=\bigcup_{z\in K}B(z,\eps)$ stands for an $\eps$-neighborhood of a set.
The set of all compact subsets of $X$ equipped with Hausdorff distance is a metric space which we will denote
by $\cK(X)$. If $X$ is a computable metric space, then $\cK(X)$ inherits this property; the ideal
points in $\cK(X)$ can be taken to be, for instance, finite unions of closed ideal balls in $X$. We then have the following:

\begin{prop}
\label{comp-set-1}
A set $K\Subset X$ is computable if and only if it is a computable point in $\cK(X)$.
\end{prop}

\begin{prop}
\label{comp-set-2}
Equivalenly, $K$ is computable if there exists an algorithm $\cA$ with a single natural input $n$, which
outputs a finite collection of closed ideal balls $\ov{B_{{1}}},\ldots, \ov{B_{i_n}}$ such that
$$\dist_H(\bigcup_{j=1}^{i_n} \ov{B_{j}}, K)<2^{-n}.$$
\end{prop}

We end this section with the following computable version of compactness. 

\begin{definition}
A set $K\subseteq X$ is called \defin{computably compact} if it is compact and there exists an algorithm $\cA$ which on input $(i_{1},\ldots ,i_{p})$ halts if and only if
 $(B_{i_{1}},\ldots ,B_{i_{p}})$ is a covering of $K$.
\end{definition}

In other words, a compact set $K$ is computably compact if we can semi-decide  the inclusion $K\subset U$, uniformly from a description of the open set $U$ as a lower computable open set. It is not hard to see that when the space $X$ is computably compact, then the collection of subsets of $X$ which are computably compact coincides with the collection of upper-computable closed sets. As an example,  it is easy to see that a  singleton $\{x\}$ is a computably compact set if and only if  $x$ is a computable point.

\subsection{Computability of probability measures}

Let $\M(X)$ denote the set of Borel probability measures over a metric space $X$, which we will assume to be endowed with a computable structure for which it is computably compact. We recall the notion of
weak convergence of measures:

\begin{definition}
A sequence of measures $\mu _{n}\in \M(X)$ is said to be \defin{weakly convergent} to $\mu\in \M(X) $ if $\int f d\mu_{n}\rightarrow \int f d\mu$ for each $f\in C_{0}(X)$. 
\end{definition}

 It is well-known, that when  $X$ is a compact separable and complete metric space, then so is $\M(X)$. In this case, weak convergence on $\M(X)$ is compatible with the notion of \defin{Wasserstein-Kantorovich distance}, defined by:

\begin{equation*}
W_{1}(\mu,\nu)=\underset{f\in 1\text{-Lip}(X)}{\sup}\left|\int f d\mu-\int f d\nu\right|
\end{equation*}
\noindent where $1\mbox{-Lip}(X)$ is the space of Lipschitz functions on $X$, having Lipschitz constant less than one.

  The following result (see \cite{HoyRoj07}) says that the computable metric structure on $X$ is inherited by $\M(X)$.  

\begin{proposition}
Let $\cD$ be the set of finite convex rational combinations of Dirac measures supported on ideal points of the computable metric space $X$. Then the triple $(\M(X),W_{1}, \cD)$ is a computable metric space. 
\end{proposition}

\begin{defn}
A \defin{computable measure} is a computable point in $(\M(X),W_{1}, \cD)$. That is, it is a measure which can be algorithmically approximated in the weak sense by \emph{discrete} measures with any given precision.  
\end{defn}

The following proposition (see \cite{HoyRoj07}) brings the previous definition to a more familiar setting. 

\begin{proposition} A probability measure $\mu$ on $X$ is computable if and only if one can uniformly compute integrals  of computable functions.  
\end{proposition}

We also note (see e.g. \cite{GalHoyRoj09}).

\begin{proposition}
\label{comp-supp}
The support of a computable measure is a lower-computable set. 
\end{proposition}

\noindent
As examples of computable measures we mention Lebesgue measure in $[0,1]^{n}$, or any smooth measure in $[0,1]^n$
with a computable density function.


\subsection{Time complexity of a problem.} For an algorithm $\cA$ with input $w$ the {\it running time}  is the number
of steps $\cA$ makes before terminating with an output.
The {\it size} of an input $w$ is the number of dyadic bits required to specify $w$. Thus for $w\in\NN$, the size of $w$ is
the smallest integer  $l(w)\geq\log_2 w$.
The {\it running time of $\cA$} is
the function
$$T_\cA :\NN \to\NN$$
 such that
$$T_\cA(n) = \max\{\text{the running time of }\cA(w)\text{ for inputs } w \text{ of size }n\}.$$
In other words, $T_\cA(n)$ is the worst case running
time for inputs of size $n$.
For a computable function $f : \NN\to \NN$  the time complexity of $f$ is said to have an
upper bound $T(n)$ if there exists an algorithm $\cA$ with running time bounded
by $T(n)$ that computes $f$. We say that the time complexity of $f$ has a lower bound $T(n)$ if for every algorithm $\cA$ which computes $f$,
there is a subsequence $n_k$ such that the running time $$T_\cA(n_k)>T(n_k).$$

\subsection{Computational complexity of two-dimensional images.}

Intuitively, we say the time complexity of 
a set $S\subset\RR^2$ is $t(n)$ if it takes time $t(n)$ to decide whether to draw a pixel of size $2^{-n}$ 
in the picture of $S$.  Mathematically, the definition is as follows:

\begin{defn}
\label{def1}
A set $T$ is said to be a $2^{-n}$-picture of a bounded set $S\subset\RR^2$ if:
\begin{itemize}
\item[(i)] $S \subset T$, and 
\item[(ii)] $T \subset B(S,2^{-n}) =
\{ x \in \R^2~:~|x-s|<2^{-n}~\mbox{for some~} s\in S\}$.
\end{itemize}
\end{defn}

\noindent
Definition \ref{def1} means that $T$ is a $2^{-n}$-approximation of $S$ with respect to 
the {\em Hausdorff metric}, given by
$$
d_H (S,T) := \inf \{ r : S \subset B(T,r)~~\mbox{and}~~
T \subset B(S,r) \}.
$$

Suppose we are trying to generate a  picture of a set $S$ using 
a union of round pixels of radius $2^{-n}$ with centers at all 
the points of the form $\left( \frac{i}{2^{n}}, 
\frac{j}{2^{n}}\right)$, with $i$ and $j$ integers. In order to 
draw the picture, we have to decide for each pair $(i,j)$ whether to 
draw the pixel centered at $\left( \frac{i}{2^{n}},
\frac{j}{2^{n}}\right)$ or not. We want to draw the pixel if 
it intersects $S$ and to omit it if some neighborhood of the pixel 
does not intersect $S$. Formally, we want to compute a function 
\beq
\label{star1}
f_S ( n, i/2^{n}, j/2^{n}) = 
\left\{ 
\begin{array}{ll}
1, & B((i/2^{n}, j/2^{n}), 2^{-n})\cap S \neq \emptyset \\
0, & B((i/2^{n}, j/2^{n}), 2 \cdot 2^{-n})\cap S = \emptyset  \\
0~~\mbox{or}~~1,~~& \mbox{in all other cases}
\end{array}
\right. 
\eeq


The time complexity of  $S$ is defined as follows.

\begin{defn}
\label{defcomp}
A bounded set $S$ is said to be computable in time $t(n)$ if there 
is 
a function $f(n,\cdot)$ satisfying \eqref{star1} which runs in time 
$t(n)$. We say that $S$ is poly-time computable if there is a polynomial 
$p$, such that $S$ is computable in time $p(n)$. 
\end{defn} 

Computability of sets in bounded space is defined in a 
similar manner. There, the amount of {\em memory} the 
machine is allowed to use is restricted. 

To see why this is the ``right" definition, suppose we are trying to 
draw a set $S$ on a computer screen which has a $1000\times 1000$ pixel
resolution. A $2^{-n}$-zoomed in picture of $S$ has $O(2^{2 n})$ pixels
of size $2^{-n}$, and thus would take time $O(t(n) \cdot 2^{2 n})$ to 
compute. This quantity is exponential in $n$, even if $t(n)$ is bounded 
by a polynomial. But we are drawing $S$ on a finite-resolution display, 
and we will only need to draw $1000 \cdot 1000 = 10^6$ pixels. Hence 
the running time would be $O(10^6 \cdot t(n)) = O(t(n))$. This running 
time is polynomial in $n$ if and only if $t(n)$ is polynomial. 
Hence $t(n)$ reflects the `true' cost of zooming in. 

\section{Computability and complexity of Conformal Mappings}
Let $\DD=\{|z|<1\}\subset\CC$. The celebrated Riemann Mapping Theorem asserts:

\begin{riemann-mapping-theorem}
Suppose $ W\subsetneq\CC$ is a simply-connected domain in the complex plane, and let $w$ be an arbitrary
point in $W$. Then there exists a unique conformal mapping
$$f:\DD\to W,\text{ such that }f(0)=w\text{ and }f'(0)>0.$$
\end{riemann-mapping-theorem}

\noindent
The inverse mapping,
$$\varphi\equiv f^{-1}: W\to\DD$$
is called the Riemann mapping of the domain $ W$ with base point $w$. The first complete proof of Riemann Mapping Theorem
was given by Osgood \cite{osgood} in 1900.
The first constructive proof of the Riemann Uniformization Theorem is due to Koebe \cite{Koebe},
and dates to the early 1900's. Formal proofs of the constructive nature of the Theorem
which follow Koebe's argument 
under various computability conditions on the boundary of the domain are numerous in the
literature (see e.g. \cite{Cheng,BB,Zhou,Hertling}).

The following theorem, due to Hertling  \cite{Hertling}, characterises the information required from the domain $W$ in order to compute the Riemann map. 

\begin{thm}
\label{comp-riem-map}
Let $W\subsetneq \CC$ be a simply-connected domain. Then, the following are equivalent:
\begin{itemize}
\item[(i)] $W$ is a lower-computable open set, $\partial W$ is a lower-computable closed set,
and $w\in W$ be a computable point;
\item[(ii)] The conformal bijection $$f:\DD\to W, \, f(0)=w,\;f'(0)>0;\,\, \text{and its inverse}\,\, \varphi\equiv f^{-1}: W\to\DD, $$
are both computable. 
\end{itemize}
\end{thm}

We now move to discussing the computational complexity of computing the Riemann map.
For this analysis we assume the domain $W$ is computable.

For $w\in W$ as above, the quantity
$1/|\varphi'(w)|$ is called the {\it conformal radius} of $W$ at $w$.
In \cite{BBY3} it was shown that even if the domain we are uniformizing 
is very simple computationally, the complexity of the uniformization
can be quite high. In fact, it might already be difficult to compute the 
conformal radius of the domain:
\begin{thm}
\label{thm:lb}
Suppose there is an algorithm $A$ that given a simply-connected domain
$W$ with a  linear-time computable boundary and an inner radius $>\frac{1}{2}$ and a number $n$ computes
the first $20n$ digits of the
conformal radius $r(W,0)$, then we can use one call to $A$ to solve any
instance of a $\#\text{\it SAT}(n)$ with a linear time overhead. 

In other words, $\ccl{\#P}$ is poly-time reducible to computing the conformal
radius of a set. 
\end{thm}

\begin{figure}[ht]
\centerline{\includegraphics[width=0.9\textwidth]{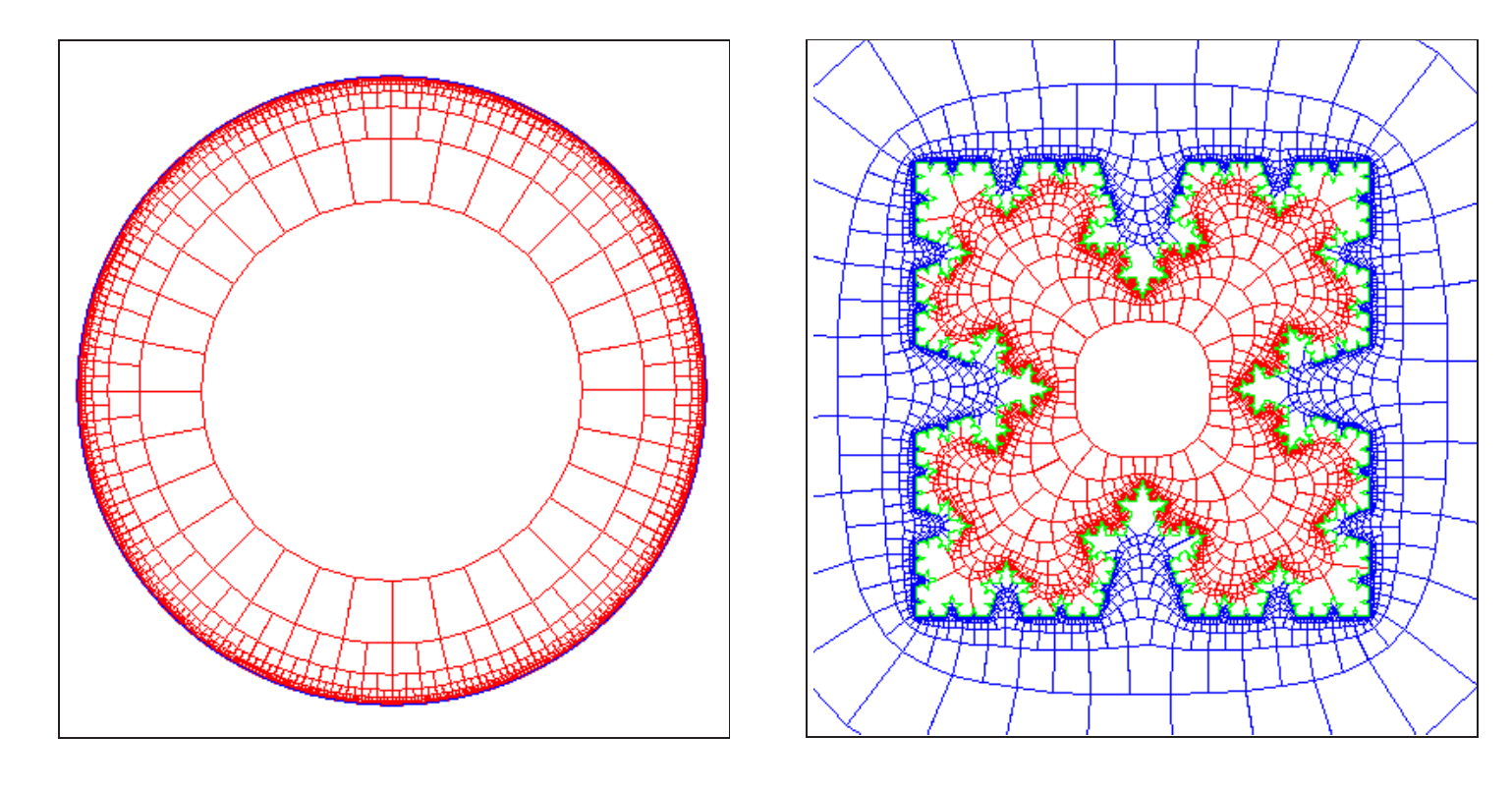}}
\caption{\label{fig-zipper}An example of the work of ``Zipper''. On the left: a standard Carleson grid on the disk. On the right: the image of the Carleson grid by the conformal map to the interior of the inverted snowflake and the image by the conformal map to the exterior of the snowflake. Both images taken from D.~Marshall's web page.
}
\end{figure}

Rettinger in \cite{Ret1} showed that this complexity bound is sharp:
\begin{thm}
  \label{riem-up}
There is an algorithm $A$ that computes the uniformizing map in
the following sense.

Let  $W$ be a bounded simply-connected domain, and $w \in W$.
Assume that the boundary $\partial W$ is gien to $A$ by an oracle (that is, $A$ can query an oracle for a function of the form (\eqref{star1})). Let $A$ also have an oracle access to $w$, and to another point $z\in W$. 
Then A computes $\varphi(z)$ with precision $n$ with complexity $\ccl{\#P}(n)$.

\end{thm}

\noindent
This result of \cite{Ret1} improved that of \cite{BBY3}, where the same statement was obtained with complexity bound $\ccl{PSPACE}$.
The statement of Theorem~\ref{riem-up} (as well as its predecessor in \cite{BBY3}) is not based on an explicit algorithm. Rather, the existence of $A$ is derived from derandomization results for random walks established in \cite{BB07}.

Computation of the mapping $\varphi$ is important for applications, and
numerous algorithms have been implemented in practice, however, none of them reaches the theoretical efficiency limit of $\ccl{\# P}$.
The most computationally
efficient algorithm used 
nowadays to calculate the conformal map is the ``Zipper'', invented by
Marshall (see \cite{Mar}). 
The effectiveness of this algorithm was studied by Marshall and Rohde in \cite{MR}. 
In the settings of the 
Theorem~\ref{riem-up}, it belongs to the complexity class $\ccl{EXP}(n)$.
It is reasonable to expect then, that an algorithm can be found in the class $\ccl{\# P}$
which is more practically efficient than ``Zipper''.

\section{Computable Carath{\'e}odory Theory}

The theory of Carath{\'e}odory (see e.g. \cite{Mil,PomUnivalent})
 deals with the question of extending the map $f$ to the unit circle. It is most widely known
in the case when $\partial W$ is a locally connected set. We remind the reader that a Hausdorff topological space $X$ is
called locally connected if for every point $x\in X$ and every open set $V\ni x$ there exists a connected set $U\subset V$
such that $x$ lies in the interior of $U$. Thus, every point $x\in X$ has a basis of connected, but not necessarily open,
neighborhoods. This condition is easily shown to be equivalent to the (seemingly stronger) requirement that every point $x\in X$
has a basis of open connected neighborhoods. In its simplest form, Carath{\'e}odory Theorem says:

\begin{baby-carat}
A conformal mapping $f:\DD\to W$ continuously extends to the unit circle if and only if $\partial W$ is locally connected.
\end{baby-carat}

\noindent
A natural question from the point of view of Computability Theory is then the following:

\medskip
\noindent
{\sl What information do we need about the boundary of the domain $ W$ in order to compute the Carath{\'e}odory extension
$f:S^1\to\partial W$ ?}

\medskip
\noindent
Below we discuss a constructive Carath{\'e}odory theory, which, in particular, answers this question.


\label{sec:carat}
\subsection{Carath{\'e}odory Extension Theorem}
We give a very brief account of the principal elements of the theory here, for details see e.g. \cite{Mil,PomUnivalent}.
In what follows, we fix a bounded simply connected domain $ W\subset\CC$, and a point $w\in W$; we will
refer to such a pair as a {\it pointed domain}, and use notation $( W,w)$ (to adopt what follows to unbounded domains, one
simply needs to replace the Euclidean metric on $\CC$ with the spherical metric on $\hat\CC$).
A {\it crosscut} $\gamma\subset W$ is a homeomorphic image
of the open interval $(0,1)$ such that the closure $\overline\gamma$ is homeomorphic to the closed inerval $[0,1]$ and the two endpoints
of $\overline\gamma$ lie in $\partial W$. It is not difficult to see that a crosscut divides
$ W$ into two connected components. Let $\gamma$ be a crosscut such that $w\notin\gamma$.
The component of $ W\setminus\gamma$ which does not contain $w$ is called {\it the crosscut neighborhood} of $\gamma$ in
$( W,w)$. We will denote it $N_\gamma$.

A {\it fundamental chain} in $( W,w)$ is a nested infinite sequence
$$N_{\gamma_1}\supset N_{\gamma_2}\supset N_{\gamma_3}\supset\cdots$$
of crosscut neighborhoods such that the closures of the crosscuts $\gamma_j$ are disjoint, and such that
$$\diam \gamma_j\longrightarrow 0.$$
Two fundamental chains $(N_{\gamma_j})_{j=1}^\infty$ and $(N_{\tau_j})_{j=1}^\infty$ are {\it equivalent} if every $N_{\gamma_j}$ contains
some $N_{\tau_i}$ and conversely, every $N_{\tau_i}$ contains
some $N_{\gamma_j}$.
Note that any two fundamental chains $(N_{\gamma_j})_{j=1}^\infty$ and $(N_{\tau_j})_{j=1}^\infty$ are either equivalent or eventually
disjoint, i.e. $N_{\gamma_j}\cap N_{\tau_i}=\emptyset$ for $i$ and $j$ sufficiently large.

The key concept of Carath{\'e}odory theory is a {\it prime end}, which is an equivalence class of fundamental
chains. The {\it impression} $\cI(p)$ of a prime end $p$ is a compact connected subset of $\partial  W$ defined as
follows: let  $(N_{\gamma_j})_{j=1}^\infty$ be any fundamental chain in the equivalence class $p$, then
$$\cI(p)=\cap\overline{N_{\gamma_j}}.$$
We say that the impression of a prime end is {\it trivial} if it consists of a single point. It is easy to see (cf. \cite{Mil}) that:
\begin{prop}
\label{trivial-impression}
If the boundary $\partial W$ is locally connected then the impression of every prime end is trivial.

\end{prop}

We define the {\it Carath{\'e}odory compactification} $\widehat  W$ to be the disjoint union of $ W$ and  the set of prime ends
of $W$ with the following topology. For any crosscut neighborhood $N$ let $\widetilde N\subset\widehat  W$ be the neighborhood
$N$ itself, and the collection of all prime ends which can be represented by fundamental chains starting with $N$. These
neighborhoods, together with the open subsets of $ W$, form the basis for the topology of $\widehat W$.
The above definition originated in \cite{Mzk}.

\begin{carat}
Every conformal isomorphism $\phi: W\to\DD$ extends uniquely to a homeomorphism
$$\hat \phi:\widehat  W\to\overline\DD.$$
\end{carat}
Carath{\'e}odory Theorem for locally connected domains is a synthesis of the above statement and \propref{trivial-impression}.

Let us note (see \cite{Mil}, p. 184):
\begin{lem}
\label{cont-lc}
If $f$ is a continuous map from a compact locally connected space $X$ onto a Hausdorff space $Y$, then $Y$ is also
locally connected.
\end{lem}

We also note:
\begin{thm}In the case when $ W$ is Jordan, the identity map $ W\to W$ extends to a
homeomorphism between the Carath{\'e}odory closure $\widehat W$ and $\overline W$.
\end{thm}

In the Jordan case, we will use the notation $\overline\phi$ for the extension of a conformal map to the closure of $ W$.
Of course,
$$\overline\phi=(\overline f)^{-1}.$$

Carath{\'e}odory compactification of $( W,w)$
can be seen as its metric completion for the following metric. Let $z_1$, $z_2$ be two points in
$ W$ distinct from $w$.
We will define the {\it crosscut distance} $\dist^ W_C(z_1,z_2)$ between $z_1$ and $z_2$ as the infimum of the diameters of
curves $\gamma$ in $ W$ for which one of the following properties holds:
\begin{itemize}
\item $\gamma$ is a crosscut such that $z_1$ and $z_2$ are contained in the crosscut neighborhood $N_\gamma$; 
\item $\gamma$ is a simple closed curve such that $z_1$ and $z_2$ are contained in the bounded component of the complement $\CC\setminus \gamma$ and 
$\gamma$ separates $z_1$, $z_2$ from $w$.
\end{itemize}

\noindent
It is easy to verify that:
\begin{thm}
The crosscut distance is a metric on $ W\setminus\{w\}$ which is locally equal to the Euclidean one.
The completion of $ W$ equipped with $\dist_C$ is homeomorphic to $\widehat W$.
\end{thm}

\subsection{Computational representation of prime ends}

\begin{defn}
We say that a curve $g:(0,1)\to\CC$ is a {\it rational polygonal curve} if:
\begin{itemize}
\item the image of $g$ is a simple curve;
\item $g$ is piecewise-linear with rational coefficients.
\end{itemize}
\end{defn}

\noindent
The following is elementary:
\begin{prop}
\label{specify-end}
Let $N_{\gamma_j}$ be a fundamental chain in a pointed simply-connected domain $( W,w)$. Then there exists
an equivalent fundamental chain $N_{\tau_j}$ such that the following holds. For every $j$ there exists a rational
polygonal curve $t_j:(0,1)\to\CC$ with
$$t_j(0.5)\in N_{\gamma_j}\setminus \overline{N_{\gamma_{j+1}}},$$
and such that $\tau_j\subset t_j([0.01,0.99])$.
Furthermore, $t_j$ can be chosen so that $$\diam t_j(0,1)\underset{j\to\infty}\longrightarrow 0.$$
\end{prop}

\noindent
We call the sequence of polygonal curves $t_j$ as described in the above proposition  a \textit{representation}
of the prime end
$p$ specified by $N_{\gamma_j}$. Since only a finite amount of information suffices to describe each rational polygonal curve $t_j$,
the sequence $t_j$ can be specified by an oracle. Namely, there exists 
an algorithm $\cA$ such that for every representation $(t_j)$ of a prime end there exists 
a function $\phi:\NN\to\NN$ such that
given access to  the values of $\phi(i),$ $i=1,\ldots,n$, the algorithm $\cA$  outputs the coefficients of the rational polygonal curves
$t_j$, $j=1,\ldots,m_n$ with $m_n\underset{n\to\infty}{\longrightarrow}\infty$. We will refer to such $\phi$ simply as {\it an oracle for }$p$.

\subsection{Structure of a computable metric space on $\widehat W$.}
Let $K\Subset\CC$. We say that $\phi$ is {\it an oracle for }$K$ if $\phi$ is a function from the natural numbers to
sets of finite sequences of triples $(x_j,y_j,r_j)$ of rational numbers with the following property.
Let $$\phi(n)=\{(x_j,y_j,r_j)\}_{j=1}^{k_n},$$
and let $B_j$ be the ball of radius $r_j$ about the point $x_j+iy_j$. Then
$$\dist_H(\bigcup_{j=1}^{k_n} \ov{ B_j}, K)<2^{-n}.$$

Let $( W,w)$ be a simply-connected pointed domain. Then the following conditional computability result holds (see \cite{BRY}):

\begin{thm}
\label{cond-comp-metric}
The following is true in the presence of oracles for $w$ and for $\partial W$.
 The Carath{\'e}odory completion $\widehat W$
equipped with the crosscut distance is a computable metric space, whose ideal points are rational points in $ W$.
Moreover, this space is computably compact.
\end{thm}

\subsection{Moduli of locally connected domains}
Suppose $\partial W$ is locally connected. The following definition is standard:
\begin{defn}[{\bf Modulus of local connectivity}]
Let $X\subset \RR^2$ is a connected set. Any strictly increasing function $m:(0,a)\to\RR$
is called a {\it modulus of local connectivity} of $X$ if
\begin{itemize}
\item for all $x,y\in X$ such that $\dist(x,y)<r<a$ there exists a connected subset $L\subset X$
containing both $x$ and $y$ with the property $\diam(L)<m(r)$;
\item $m(r)\searrow 0$ as $r\searrow 0$.
\end{itemize}
\end{defn}

\noindent
Of course, the existence of a modulus of local connectivity implies that $X$ is locally connected.
Conversely, every compact connected and locally connected set has a modulus of local connectivity.

We note that every modulus of local connectivity is also a modulus of path connectivity:
\begin{prop}
Let $m(r)$ be a modulus of local connectivity for a connected set $X\subset\RR^2$. Let $x,y\in X$
such that $\dist (x,y)<r$. Then there exists a path $\ell$ between $x$ and $y$ with diameter
at most $m(r)$.
\end{prop}
For the proof, see Proposition 2.2 of \cite{Orsay}.  The following notion was introduced in \cite{BRY}. 

\begin{defn}[{\bf Carath{\'e}odory modulus}] Let $( W,w)$ be a pointed simply-connected domain. A strictly increasing function
$\eta:(0,a)\to\RR$ is called a {\it Carath{\'e}odory modulus} if
 for every crosscut $\gamma$ with $\diam(\gamma)<r<a$ we have $\diam N_\gamma<\eta(r)$.
\end{defn}
We note (see e.g. \cite{BRY}) that this alternative modulus also allows to characterise local connectivity.  
\begin{prop}
There exists a Carath{\'e}odory modulus  $\eta(r)$ such that $\eta(r)\searrow 0$ when $r\searrow 0$
 if and only if the boundary $\partial W$ is locally connected.
\end{prop}

\subsection{Computable Carath{\'e}odory Theory}
 To simplify the exposition, we present the results for bounded domains only.
However, all the theorems we formulate below may be stated for general simply-connected domains on the 
Riemann sphere $\riem$. In this case,
the spherical metric on $\riem$ would have to be used in the statements instead of the Euclidean one.

The following is shown in \cite{BRY}:
\begin{thm}
\label{main1}
Suppose $( W,w)$ is a bounded simply-connected pointed domain. Suppose the Riemann mapping
$$\phi: W\to\DD\text{ with }\phi(w)=0,\;\phi'(w)>0$$
is computable.
Then there exists an algorithm $\cA$ which, given a representation of a prime end $p\in\widehat W$
computes the value of $\hat\phi(p)\in S^1.$

\end{thm}

\noindent
In view of \thmref{comp-riem-map},  we have:
\begin{cor}
\label{cormain1}
Suppose we are given oracles for $W$ as a lower-computable open set,  for $\partial W$ as a lower-computable closed set,
and an oracle for the value of $w$ as well.
Given a representation of a prime end $p\in\widehat W$, the value $\hat\phi(p)\in S^1$ is uniformly computable.
\end{cor}

\noindent
To state a ``global'' version of the above computability result, we use the structure of a computable metric space:
\begin{thm}[\cite{BRY}]
\label{main2}
In the presence of oracles for $w$ and for $\partial W$, both the Carath{\'e}odory extension
$$\hat\phi:\widehat W\to\overline\DD\text{ and its inverse }\hat f\equiv \hat\phi^{-1}:\overline\DD\to\widehat W$$
are computable, as functions between computable metric spaces.
\end{thm}

\begin{rem}
The assumptions of \thmref{main2} are stronger than those of \corref{cormain1}: computability of $\partial W$
implies lower computability of $W$ and $\partial W$, but not vice versa.
\end{rem}

\noindent
For the particular but important case when the domain $ W$ has a locally connected boundary, it is natural to ask what boundary information is required to make  the  extended map $\overline f:\overline\DD\to\overline W$  computable.  A natural candidate to consider is a description of the modulus of local connectivity.  Indeed, as it was shown in \cite{mcn, mcnicholl2014computing}, a computable local connectivity modulus implies computability of the Carath\'eodory extension.  Such a modulus of local connectivity, however,  turned out to be unnecessary, as shown by the following two results proven  in  \cite{BRY}.  The first one says that the right  boundary information to consider is  Carath\'eodory modulus $\eta(r)$, and the second one tells us that the two moduli, although classically equivalent, are indeed computationally different. 
\begin{thm}
\label{main3}
Suppose $( W,w)$ is a pointed simply-connected bounded domain with a locally connected boundary. Assume that  the holomorphic bijection
$$f:\DD\to W\text{ with }f(0)=w,\; f'(0)>0$$
is computable.

Then the boundary extension
$$\overline f:\overline\DD\to\overline W$$
is computable if and only if there exists a computable Carath{\'e}odory modulus $\eta(r)$ with
$\eta(r)\searrow 0$ as $r\searrow 0$.
\end{thm}

\begin{remark}
With routine modifications, the above result can be made \emph{uniform} in the sense that there is an algorithm which from a description of $f$ and $\eta$ computes a description of $\overline f$, and there is an algorithm which from a description of $\overline f$
computes a Carath{\'e}odory modulus $\eta$. See for example  \cite{Hertling}  for statements made in this generality.
\end{remark}

We note that the seemingly more ``exotic'' Carath{\'e}odory modulus cannot be replaced by the modulus of local connectivity in the above
statement:
\begin{thm}[\cite{BRY}]
\label{main4}
There exists a simply-connected domain $ W$ such that $\partial W$ is  locally-connected, $\overline W$ is computable,
and there exists a computable Carath{\'e}odory modulus $\eta(r)$, however, no computable modulus of local connectivity exists
for $\partial W$.
\end{thm}


Finally, we turn to computational complexity questions in the cases when $\overline \phi$ or $\overline f$ are computable. Intuitively, when the boundary of the domain has a geometrically complex structure, one expects the Carath\'eodory extension to also be computationally complex.  Using this idea,  in \cite{BRY} it is shown that:
\begin{thm}
\label{main5}
Let $q:\NN\to\NN$ be any computable function. There exist Jordan domains $ W_1\ni 0$, $ W_2\ni 0$ such that the following holds:
\begin{itemize}
\item the closures $\overline W_1$, $\overline W_2$ are computable;
\item the extensions $\overline\phi:\overline W_1\to\overline\DD$ and  $\overline f:\overline \DD\to\overline W_2$  are both computable functions;
\item the time complexity of $\overline f$ and $\overline\phi$ is bounded from below by $q(n)$ for large enough values of $n$.
\end{itemize}

\end{thm}

In other words, the computational complexity of the extended map can be arbitrarily high. 
\section{Computability in Complex Dynamics: Julia sets}\label{juliasets}

\nothing{
  
\subsection{Numerical simulation of a chaotic dynamical system: the modern paradigm}

A dynamical system can be simple, and thus easy to implement numerically. Yet its orbits may exhibit a very complex
behaviour. The famous paper of Lorenz \cite{Lorenz}, for example, described a rather simple nonlinear system of
ordinary differential equations $\bar x'(t)=F(\bar x)$ in three dimension which exhibits {\it chaotic} dynamics.
In particular, while the flow of the system $\Phi_t(\bar x_0)$ is easy to calculate with an arbitrary precision for 
any initial value $x_0$ and any time $t$. However, any error in estimating the initial value $\bar x_0$ grows exponentially
with $t$. This renders impractical attempting to numerically simulate the behaviour of a trajectory of the system for an extended 
period of time: small computational errors are magnified very rapidly. If we recall that the Lorenz system was introduced
as a simple model of weather forecasting, one understands why predicting weather conditions several days in advance
is difficult to do with any accuracy. 

On the other hand, there is a great regularity in the global structure of a typical trajectory of Lorenz system. As was
ultimately shown by Tucker \cite{Tucker}, there exists a set $\cA\subset \RR^3$ such that for almost every initial point $\bar x_0$, the 
limit set of the orbit,
$$\omega(\bar x_0)=\cA.$$
This set is the {\it attractor} of the system \cite{Smale,Milnor-attractor}. Moreover, for any continuous test function $\psi$, the 
time average of $\psi$ along a typical orbit 
$$\frac{1}{T}\int_{t=0}^T \psi(\Phi_{\bar x_0}(t) dt$$
converges to the integral $\int \psi d\mu$ with respect to a measure $\mu$ supported on $\cA$. 

Thus, both the spatial layout and the statistical properties of a large segment of a typical trajectory can be understood, and,
indeed, simulated on a computer: even mathematicians unfamiliar with dynamics have seen the butterfly-shaped picture of  Lorenz attractor $\cA$.
This example summarizes the modern paradigm of numerical study of chaos: while the simulation of an individual orbit for an extended
period of time does not make a practical sense, one should study the limit set of a typical orbit (both as a spatial object and as
a statistical distribution).
A modern summary of this paradigm is found, for example, in the article of J.~Palis \cite{Palis}.

}

\subsection{Basic properties of Julia sets}
An excellent general reference for the material in this section is 
the textbook of Milnor \cite{Mil}.

The modern paradigm of numerical study of chaos (see e.g. the article of J.~Palis \cite{Palis}) can be briefly summarized as follows.
While the simulation of an individual orbit for an extended
period of time does not make a practical sense, one should study the limit set of a typical orbit (both as a spatial object and as
a statistical distribution).
One of the best known illustration of this approach is the numerical study of Julia sets in Complex Dynamics. The Julia set $J(R)$ is the {\it repeller} of a rational mapping 
 $R$ of degree $\deg R=d\geq 2$ considered as a dynamical system on the Riemann
sphere
$$R:\riem\to\riem,$$
that is, it is the attractor for (all but at most two) inverse orbits under $R$
$$z_0,z_{-1},z_{-2},\ldots,\text{ where }R(z_{-n+1})=z_{-n}.$$
In fact, let $\delta_z$ denote the Dirac measure at $z\in\riem$.
Denote $$\mu_n(z)=\frac{1}{d^n}\sum_{w\in R^{-n}(z)}\delta_w.$$
These probability  measures assign equal weight to the $n$-th preimages of $z$, counted with multiplicity. 
Then for all points $z$ except at most two, the measures $\mu_n(z)$ weakly converge to the {\it Brolin-Lyubich measure} $\lambda$ of $R$, whose support is equal to $J(R)$ \cite{Brolin,Lyubich}.

Another way to define $J(R)$ is as the locus of chaotic dynamics of $R$, that is, 
the complement of the set where the dynamics is Lyapunov-stable:

\begin{defn}
Denote $F(R)$ the set of points $z\in\riem$ having an open neighborhood $U(z)$ on which the
family of iterates $R^n|_{U(z)}$ is equicontinuous. The set $F(R)$ is called the Fatou set of $R$
and its complement $J(R)=\riem\setminus F(R)$ is the Julia set.
\end{defn}

\noindent
Finally, when the rational mapping is a polynomial $$P(z)=a_0+a_1z+\cdots+a_dz^d:\CC\to\CC$$ an equivalent
way of defining the Julia set is as follows. Obviously, there exists a neighborhood of $\infty$ on $\riem$
on which the iterates of $P$ uniformly converge to $\infty$. Denoting $A(\infty)$ the maximal such domain of attraction
of $\infty$ we have $A(\infty)\subset F(R)$. We then have 
$$J(P)=\partial A(\infty).$$
The bounded set $\riem \setminus  A(\infty)$ is called {\it the filled Julia set}, and denoted $K(P)$;
it consists of points whose orbits under $P$ remain bounded:
$$K(P)=\{z\in\riem|\;\sup_n|P^n(z)|<\infty\}.$$
Using the above definitions, it is not hard to make a connection between computability questions for polynomial Julia sets, and the general framework of computable complex analysis discussed above. For instance, the Brolin-Lyubich measure $\lambda$ on the Julia set is the {\it harmonic measure} at $\infty$ of the basin $A(\infty)$; that is, the pull-back of the Lebesgue measure on the unit circle by the appropriately normalized Riemann mapping of the basin $A(\infty)$. We will discuss the computability of this measure below. We will also see an even more direct connection to computability of a Riemann mapping when we talk about computability of polynomial Julia sets.

\noindent
For future reference, let us summarize in a proposition below the main properties of Julia sets:

\begin{prop}
\label{properties-Julia}
Let $R:\riem\to\riem$ be a rational function. Then the following properties hold:
\begin{itemize}
\item $J(R)$ is a non-empty compact subset of $\riem$ which is completely
invariant: $R^{-1}(J(R))=J(R)$;
\item $J(R)=J(R^n)$ for all $n\in\NN$;
\item $J(R)$ has no isolated points;
\item if $J(R)$ has non-empty interior, then it is the whole of $\riem$;
\item let $U\subset\riem$ be any open set with $U\cap J(R)\neq \emptyset$. Then there exists $n\in\NN$ such that
$R^n(U)\supset J(R)$;
\item periodic orbits of $R$ are dense in $J(R)$.
\end{itemize}
\end{prop}

\noindent
For a periodic point $z_0=R^p(z_0)$
of period $p$ its {\it multiplier} is the quantity $\lambda=\lambda(z_0)=DR^p(z_0)$.
We may speak of the multiplier of a periodic cycle, as it is the same for all points
in the cycle by the Chain Rule. In the case when $|\lambda|\neq 1$, the dynamics
in a sufficiently small neighborhood of the cycle is governed by the Mean
Value Theorem: when $|\lambda|<1$, the cycle is {\it attracting} ({\it super-attracting}
if $\lambda=0$), if $|\lambda|>1$ it is {\it repelling}.
Both in the attracting and repelling cases, the dynamics can be locally linearized:
\begin{equation}
\label{linearization-equation}
\psi(R^p(z))=\lambda\cdot\psi(z)
\end{equation}
where $\psi$ is a conformal mapping of a small neighborhood of $z_0$ to a disk around $0$.

\noindent
In the case when $|\lambda|=1$, so that $\lambda=e^{2\pi i\theta}$, $\theta\in\RR$, 
 the simplest to study is the {\it parabolic case} when $\theta=n/m\in\QQ$, so $\lambda$ 
is a root of unity. In this case $R^p$ is not locally linearizable and  $z_0\in J(R)$.

 In the complementary situation, two non-vacuous possibilities  are considered:
{\it Cremer case}, when $R^p$ is not linearizable, and {\it Siegel case}, when it is.
In the latter case, the linearizing map $\psi$ from (\ref{linearization-equation}) conjugates
the dynamics of $R^p$ on a neighborhood $U(z_0)$ to the irrational rotation by angle $\theta$
(the {\it rotation angle})
on a disk around the origin. The maximal such neighborhood of $z_0$ is called a {\it Siegel disk}.

Fatou showed that for a rational mapping $R$ with $\deg R=d\geq 2$
at most finitely
many periodic orbits are non-repelling. A sharp bound on their number  depending on $d$ has
been established by Shishikura; it is equal to the number of critical points of $R$ counted with
multiplicity: 

\begin{fshb}
For a rational mapping of degree $d$ the number of the non-repelling periodic
cycles taken together with the number of cycles of Herman rings is at most $2d-2$.
For a polynomial of degree $d$ the number of non-repelling periodic cycles in $\CC$
is at most $d-1$.
\end{fshb}

\noindent
 Therefore, the last statement of \propref{properties-Julia} can be restated as:

\begin{itemize}
\item {\it repelling periodic orbits are dense in $J(R)$}.
\end{itemize}

To conclude the discussion of the basic properties of Julia sets, let us consider the simplest
examples of non-linear rational endomorphisms of the Riemann sphere, the quadratic polynomials.
Every affine conjugacy class of quadratic polynomials has a unique representative of the
form $f_c(z)=z^2+c$, the family
$$f_c(z)=z^2+c,\;c\in\CC$$
is often referred to as {\it the quadratic family}.
For a quadratic map the structure of the Julia set is governed by the behavior of the orbit of the only
finite critical point $0$. In particular, the following dichotomy holds:

\begin{prop}
\label{quadratic-Julia}
Let $K=K(f_c)$ denote the filled Julia set of $f_c$, and $J=J(f_c)=\partial K$. Then:
\begin{itemize}
\item $0\in K$ implies that $K$ is a connected, compact subset of the plane with connected complement;
\item $0\notin K$ implies that $K=J$ is a planar Cantor set.
\end{itemize}
\end{prop}

\noindent
The {\it Mandelbrot set} $\cM\subset \CC$ is defined as the set of parameter values $c$ for which 
$J(f_c)$ is connected.

\noindent
A rational mapping $R:\hat\CC\to\hat\CC$ is called {\it hyperbolic} if the orbit of every critical point of $R$
is either periodic, or  converges to an
attracting cycle.
As easily follows from Implicit Function Theorem and considerations of local dynamics of an attracting orbit,
hyperbolicity is an open property in the parameter space of rational mappings of degree $d\geq 2$.

\noindent
Considered as a rational mapping of the Riemann sphere, a quadratic polynomial $f_c(z)$ has two critical points:
the origin, and the super-attracting fixed point at $\infty$. In the case when $c\notin \cM$, the orbit of the
former converges to the latter, and thus $f_c$ is hyperbolic. 
A classical result of Fatou implies that whenever $f_{c}$ has an attracting orbit in $\CC$, this orbit attracts the orbit of the critical point. Hence,
$f_c$  is a hyperbolic
mapping and $c\in \cM$. The following conjecture is central to the field of dynamics in one complex variable:

\medskip
\noindent
{\bf Conjecture (Density of Hyperbolicity in the Quadratic Family).} {\it Hyperbolic parameters are dense  
in $\cM$.}

\medskip
\noindent
Fatou-Shishikura Bound implies that a quadratic polynomial has at most one non-repelling cycle in the complex
plane. Therefore, 
we will call the polynomial $f_c$ (the parameter $c$, the Julia set $J_c$) {\it Siegel, Cremer,} or 
{\it parabolic} when it has an orbit of the corresponding type.

\subsection{Occurence of Siegel disks and Cremer points in the quadratic family}
\label{sec:quad siegel}
Before stating computability/complexity results for Julia sets we need to
discuss in more detail the occurrence of Siegel disks in the quadratic family.
For a number $\theta\in [0,1)$ denote $[r_1,r_2,\ldots,r_n,\ldots]$, $r_i\in\NN\cup\{\infty\}$ its possibly finite 
continued fraction expansion:
\begin{equation}
\label{cfrac}
[r_1,r_2,\ldots,r_n,\ldots]\equiv\cfrac{1}{r_1+\cfrac{1}{r_2+\cfrac{1}{\cdots+\cfrac{1}{r_n+\cdots}}}}
\end{equation}
Such an expansion is defined uniquely if and only if $\theta\notin\QQ$. In this case, the {\it rational 
convergents } $p_n/q_n=[r_1,\ldots,r_{n}]$ are the closest rational approximants of $\theta$ among the
numbers with denominators not exceeding $q_n$.

\nothing{
\begin{defn}
The {\it diophantine numbers
of order k}, denoted $\cD(k)$
is the following class of irrationals ``badly'' approximated by rationals.
 By definition, $\theta\in\cD(k)$ if there exists $c>0$ such that
$$q_{n+1}<cq_n^{k-1}$$
\end{defn}

\noindent
The numbers $q_n$ can
be calculated from the recurrent relation
$$q_{n+1}=r_{n+1}q_n+q_{n-1},\text{ with }q_0=0,\; q_1=1.$$ Therefore, $\theta\in\cD(2)$ if and only if the sequence $\{r_i\}$
is bounded. Dynamicists call such numbers {\it bounded type} (number-theorists prefer {\it constant type}). 
An extreme example of a number of bounded type is the golden mean
$$\theta_*=\frac{\sqrt{5}-1}{2}=[1,1,1,\ldots].$$
The set $$\displaystyle\cD(2+)\equiv\bigcap_{k>2}\cD_k$$
has full measure in the interval $[0,1)$. In 1942 Siegel showed:

\begin{thm}[\cite{siegel}]
Let $R$ be an analytic map with a periodic point $z_0\in\riem$ of period $p$. Suppose the multiplier of the cycle
$$\lambda=e^{2\pi i\theta}\text{ with }\theta\in\cD(2+),$$
then the local linearization equation (\ref{linearization-equation}) holds.
\end{thm}

}

\noindent
Inductively define $\theta_1=\theta$
and $\theta_{n+1}=\{1/\theta_n\}$. In this way, 
$$\theta_n=[r_{n},r_{n+1},r_{n+2},\ldots].$$
We define the {\it Yoccoz's Brjuno function} as
$$\Phi(\theta)=\displaystyle\sum_{n=1}^{\infty}\theta_1\theta_2\cdots\theta_{n-1}\log\frac{1}{\theta_n}.$$
In 1972, Brjuno proved the following:
\begin{thm}[\cite{Bru}]
Let $R$ be an analytic map with a periodic point $z_0\in\riem$ of period $p$. Suppose the multiplier of the cycle
$$\lambda=e^{2\pi i\theta}\text{ with }\Phi(\theta)<\infty,$$
then the local linearization equation (\ref{linearization-equation}) holds.
\end{thm}
This theorem generalized the classical result by Siegel \cite{siegel} which established the same statement for all Diophanitive values of $\theta$.

\begin{figure}[h]
\centerline{\includegraphics[width=0.5\textwidth]{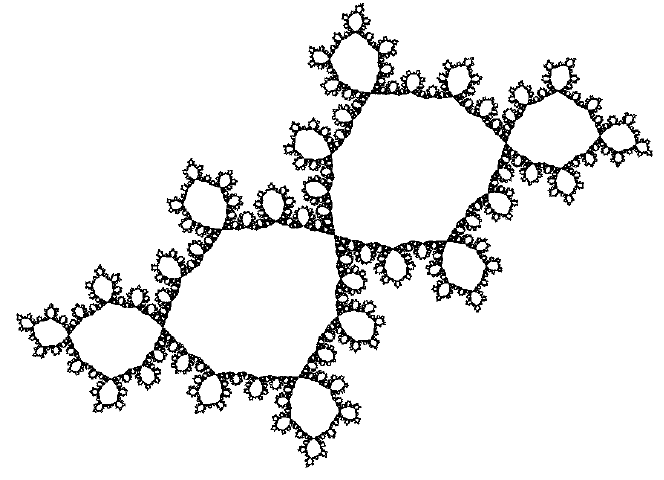}}
\caption{The  Julia set of $P_\theta$ for $\theta=[1,1,1,1,\ldots]$ (the inverse golden mean). This set is actually computable (see \cite{BY-MMJ}).}
\label{fig-siegel}
\end{figure}

\noindent
Note that a quadratic polynomial with a fixed Sigel disk with rotation angle $\theta$ after an affine
change of coordinates can be written as 
\begin{equation}
\label{form-1}
P_\theta(z)=z^2+e^{2\pi i \theta}z.
\end{equation}
\noindent
In 1987 Yoccoz \cite{Yoc} proved the following converse to Brjuno's Theorem:

\begin{thm}[\cite{Yoc}]
Suppose that for $\theta\in[0,1)$ the polynomial $P_\theta$ has a Siegel point at the origin.
Then $\Phi(\theta)<\infty$.
\end{thm}

\noindent
In fact, the value of the function $\Phi$ is directly related to the size of the Siegel disk in the following way.

\begin{defn}
Let $P(\theta)$ be a quadratic polynomial with a Siegel disk $\Delta_\theta\ni 0$. Consider a conformal isomorphism
$\phi:\DD\mapsto\Delta$ fixing $0$. The {\it conformal radius of the Siegel disk $\Delta_\theta$} is
the quantity
$$r(\theta)=|\phi'(0)|.$$
For all other $\theta\in[0,\infty)$ we set $r(\theta)=0$. 
\end{defn} 

\noindent
By the Koebe One-Quarter Theorem of classical complex analysis, the internal radius of $\Delta_\theta$ is at least
$r(\theta)/4$. Yoccoz \cite{Yoc} has shown that the sum 
$$\Phi(\theta)+\log r(\theta)$$
is bounded from below independently of $\theta\in\cB$. Buff and Ch{\'e}ritat have greatly improved this result
by showing that:

\begin{thm}[\cite{BC}]
\label{phi-cont}
The function $\upsilon:\theta\mapsto \Phi(\theta)+\log r(\theta)$ extends to $\RR$ as a 1-periodic continuous
function.
\end{thm}

\noindent
The following stronger conjecture exists (see \cite{MMY}):

\medskip
\noindent
{\bf Marmi-Moussa-Yoccoz Conjecture.} \cite{MMY} {\it The function $\upsilon: \theta\mapsto \Phi(\theta)+\log r(\theta)$ is H{\"o}lder of exponent $1/2$.}

\medskip
\noindent
It is important to note that computability of the function $\upsilon:\RR\to\RR$ would follow from above conjecture. In fact, the following is true
(see \cite{BY-book}):

\begin{tet}
If the function 
$$\upsilon: \theta\mapsto \Phi(\theta)+\log r(\theta)$$
has a computable modulus of continuity, then it 
is uniformly computable on the entire interval $[0,1]$. 
\end{tet}

\begin{lem}
[{\bf Conditional}] \label{2main-equiv}
Suppose the Conditional Implication holds. Let $\theta \in [0,1]$
be such that $\Phi(\theta)$ is finite. Then there is an oracle Turing 
Machine $M^{\phi}_1$ computing $\Phi(\theta)$ with an oracle access to 
$\theta$ if and only if there is an oracle Turing 
Machine $M^{\phi}_2$ computing $r(\theta)$ with an oracle access to 
$\theta$. 
\end{lem}

\begin{figure}[ht]
\label{phi figure}
\centerline{
\mbox{\includegraphics[width=0.45\textwidth]{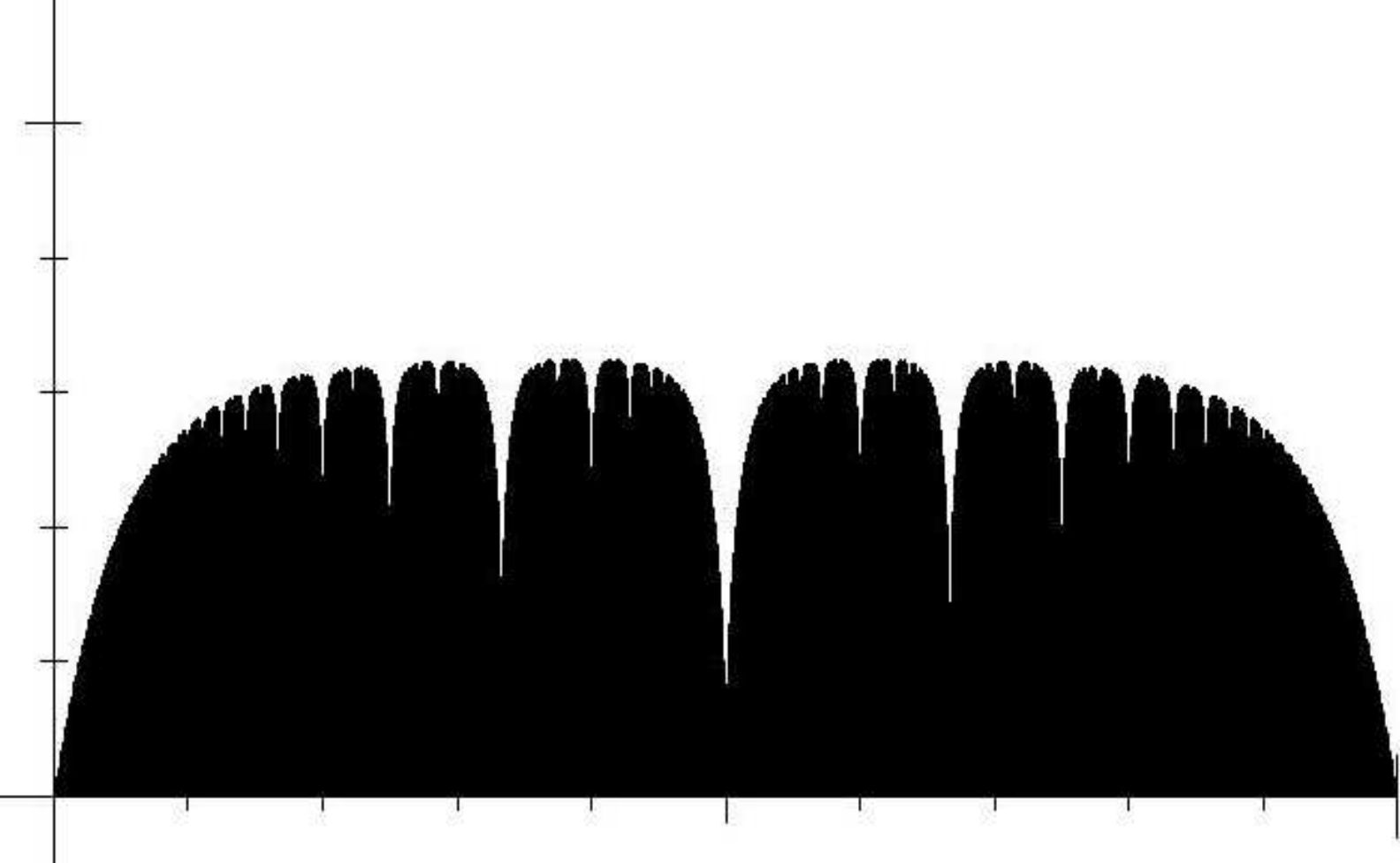}}
\mbox{\includegraphics[width=0.45\textwidth]{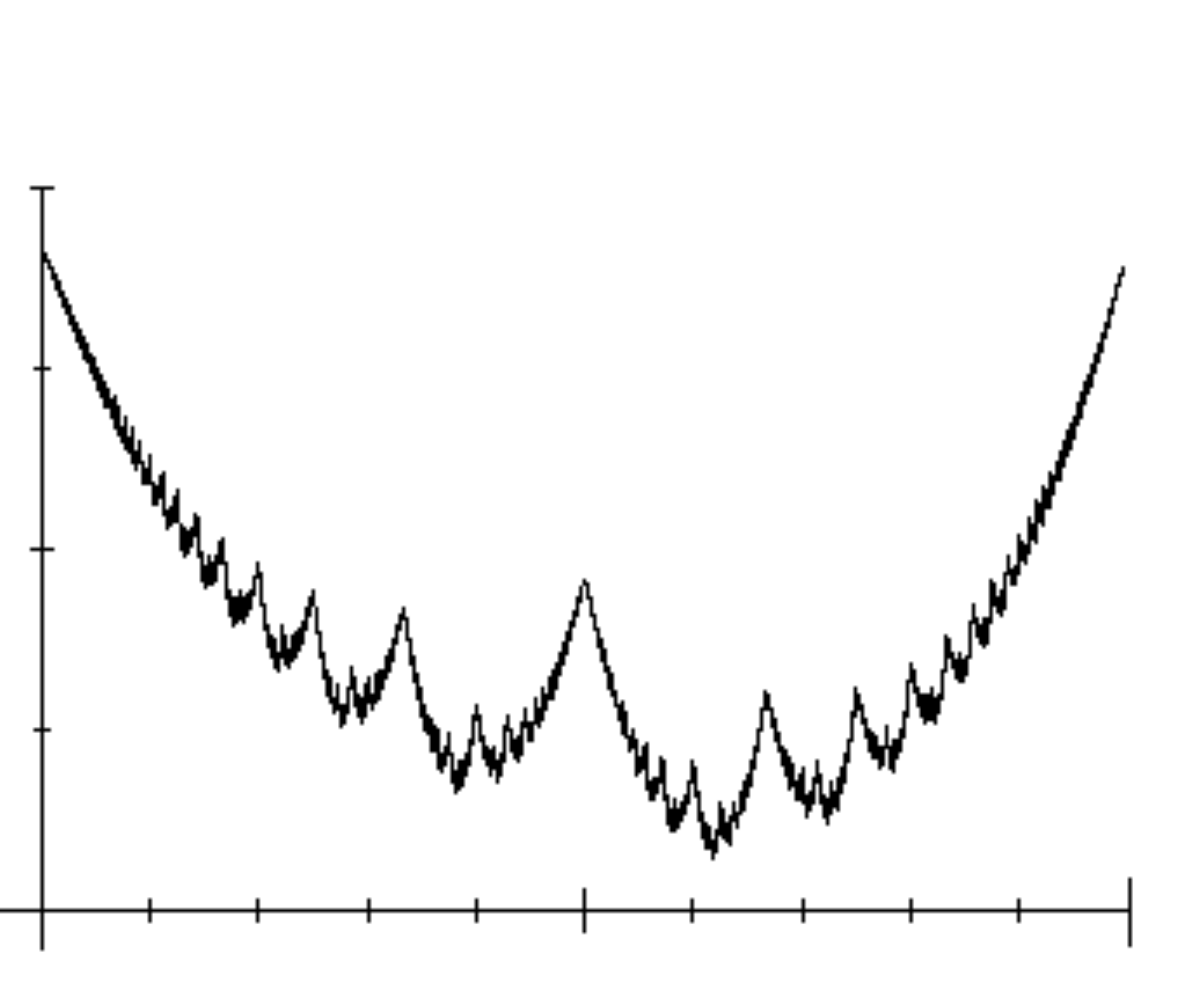}}
}
\caption{The figure on the left is an attempt to visualize the function $\Phi$,
by plotting the heights of $\exp(-\Phi(\theta))$ over a grid of Brjuno irrationals.
On the right is the graph of the (conjecturally computable) function $\upsilon(x)$.}

{\it Both figures courtesy of Arnaud Ch{\'e}ritat}

\end{figure}

\subsection{Computability of Julia sets}
When we address computability of a Julia set of a rational map $R$, it will always be by a Turing machine $M^\theta$ with an oracle access to the coefficients of the rational map $R$. This corresponds to the numerical problem of drawing a Julia set when the coefficients of the map $R$ are given with an arbitrary finite precision.

The question of computability of polynomial Julia sets was first raised in a paper of Zhong \cite{Zh}.
I was investigated exhaustively by Braverman and Yampolsky (see the monograph \cite{BY-book} as a general reference) with surprising results. For simplicity of exposition let us specialize to the case of the quadratic family.
Firstly,
\begin{thm}[\cite{BY-MMJ}]
  \label{comput-filled}
The filled Julia set $K_c$ is always computable by an oracle  Turing machine $M^\theta$ with an oracle for $c$.  

\end{thm}
The special case of this theorem when $K_c$ has empty interior was done in \cite{BBY1}. It is not hard to give an idea of the proof under this restriction. Indeed, the basin $A(\infty)$ is clearly lower-computable: simply take a large enough $R=R(c)$ such that for $|z|>R$, we have $|f_c(z)|>2R$. Then the basin $A(\infty)$ is exhausted by the countable union of the computable (with an oracle for $c$) open sets
$$f_c^{-n}(\{|z|>R\})\subset \riem\;;n\in\NN.$$
On the other hand, $K_c=J_c$ is a lower-computable closed set. It can be saturated by a countable sequence of computable (again with an oracle for $c$)
finite sets
$$W_k\equiv\{\text{repelling periodic orbits of }f_c\text{ with period }\leq k\}, \; k\in\NN.$$
A simpler algorithm for lower-computing $J_c$, and something that is actually used in practice, is to find a single fixed point $p\in J_c$ (elementary considerations imply that from the two fixed point of $f_c$ counted with multiplicity, at least one is either repelling or parabolic) and then
saturate $J_c$ by the sequence
$$W'_k\equiv \cup_{j\leq k}f_c^{-j}(p).$$
The proof of  the general case of Theorem~\ref{comput-filled} is rather more involved.

The next statements will directly relate computability of Julia sets with computable Riemann mapping:
\begin{thm}[\cite{BY}]
  \label{th-comp1}
Suppose $f_c$ has a Siegel periodic point $p$ and let $\Delta\ni p$ be the corresponding Siegel disk. Then $J_c$ is computable by a Turing machine $M^\phi$ with an oracle for $c$ if and only if the conformal radius of $\Delta$ is computable by a Turing machine with an oracle for  $c$. 
\end{thm}

\begin{thm}[\cite{BY}]
  \label{th-comp2}
  Suppose $f_c$ has no Siegel points. Then the Julia set $J_c$ is computable by a Turing machine $M^\phi$ with an oracle for $c$.
  \end{thm}

Let us now specialize further to the case of the polynomials $P_\theta(z)=e^{2\pi i\theta}z+z^2$ with a neutral fixed point at the origin.
The main result of Braverman and Yampolsky is the following:
\begin{thm}[\cite{BY-MMJ}]
  \label{noncomput-Julia}
  There exist computable values of $\theta$ (in fact, computable by an explicit, although very complicated, algorithm) such that
  $J_\theta$ is not computable.

\end{thm}

In view of Theorem~\ref{th-comp1}, this is equivalent to the fact that the conformal radius $r_\theta$ is not computable.
Assuming Conditional Implication, this would also be equivalent to the non-computability of the value of $\Phi(\theta)$. In fact, computable values of $\theta$ for which $\Phi(\theta)$ cannot be computed (unconditionally) can also be constructed (see \cite{BY-book}).
In fact, assuming Conditional Implication, it is shown in \cite{BY-MMJ} that $\theta$ is Theorem~\ref{noncomput-Julia} can be poly-time. 

We thus observe a surprising scenario: for above values of $\theta$ a finite inverse orbit of a point can be effectively, and possibly efficiently, computed with an arbitrary precision. Yet the repeller $J_\theta$ cannot be computed at all. This serves as a cautionary tale for applications of the numerical paradigm described above.

Fortunately, the phenomenon of non-computability of $J_c$ is quite rare. Such values of $c$ have Lebesgue measure zero. It is shown in \cite{BY-book}, that assuming Conditional Implication they have linear measure (Hausdorff measure with exponent $1$) zero -- a very meager set in the complex plane  indeed. Furthermore, Conditional Implication and some high-level theory of Diophantine approximations(\cite{BY-book}) imply that such values of $c$ cannot be algebraic, so even if they are easy to compute, they are not easy to write down.

It is worth noting that in \cite{BY-lc}, Braverman and Yampolsky constructed non-computable quadratic Julia sets which are locally connected. In view of the above discussed theory, for such maps, the basin of infinity
$A(\infty)=\hat\CC\setminus K_c$ is a natural example of a simply-connected domain on the Riemann sphere
with locally connected boundary such that the Riemann map is computable, but the Carath{\'e}odory extension is not (the boundary does not have a computable Carath{\'e}odory modulus).

\subsection{Computational complexity of Julia sets}
While the computability theory of polynomial Julia sets appears complete, the study of computational complexity of computable Julia sets offers many unanswered questions. Let us briefly describe the known results. As before, in all of them the Julia set of a rational function $R$ is computed by a Turing Machine $M^\phi$ with an oracle for the coefficients of $R$.

The following theorem is independently due to Braverman \cite{thesis} and Rettinger \cite{Ret}:

\begin{thm}
  \label{hyperb-polytime}
Every hyperbolic Julia set is poly-time.

\end{thm}
We note that the poly-time algorithm described in the above papers has been known to practitioners as {\it Milnor's Distance Estimator} \cite{Milnor-89}. Specializing again to the quadratic family $f_c$, we note that Distance Estimator becomes very slow (exp-time) for the values of $c$ for which $f_c$ has a parabolic periodic point. This would appear to be a natural class of examples to look for a lower complexity bound. However, surprisingly, Braverman \cite{braverman-parabolic} proved:

\begin{thm}
Parabolic quadratic Julia sets are poly-time.
\end{thm}
The algorithm presented in \cite{braverman-parabolic} is again explicit, and easy to implement in practice -- it is a major improvement over Distance Estimator.

On the other hand, Binder, Braverman, and Yampolsky \cite{BBY2} proved:
\begin{thm}
There exist Siegel quadratics of the form $P_\theta(z)=e^{2\pi i\theta}z+z^2$ whose Julia set have an arbitrarily high time complexity.
\end{thm}
Given a lower complexity bound, such a $\theta$ can be produced constructively.

A major open question is the complexity of quadratic Julia sets with Cremer points. They are notoriously hard to draw in practice; no high-resolution pictures have been produced to this day -- and yet we do not know whether any of them are computably hard.

Let us further specialize to real quadratic family $f_c$, $c\in\RR$. In this case, it was recently proved by Dudko and Yampolsky \cite{DudYam2} that:
\begin{thm}
Almost every real quadratic Julia set is poly-time.

\end{thm}
This means that poly-time computability is a ``physically natural'' property in real dynamics. Conjecturally, the main technical result of \cite{DudYam2} should imply the same statement for complex parameters $c$ as well, but the conjecture in question (Collet-Eckmann parameters form a set of full measure among non-hyperbolic parameters) while long-established, is stronger than Density of Hyperbolicity Conjecture, and is currently out of reach.

It is also worth mentioning in this regard that the most non-hyperbolic examples in real dynamics are infinitely renormalizable quadratic polynomials. The archetypic such example is the celebrated Feigenbaum polynomial. In a different paper, Dudko and Yampolsky \cite{DudYam1} showed:
\begin{thm}
The Feigenbaum Julia set is poly-time.
\end{thm}
The above theorems raise a natural question whether {\it all} real quadratic Julia sets are poly-time (the examples of \cite{BBY2} cannot have real values of $c$). That, however unlikely it appears, is unknown at present.

\subsection{Computing Julia sets in statistical terms}
As we have seen above, there are instances when for a rational map $R$, $d=\deg R\geq 2$, the set of limit points of the sequence of inverse images $R^{-n}(z)$
cannot be accurately simulated on a computer even in the ``tame'' case when $R\equiv f_c$ with a computable value of $c$.
However, even in these cases we can ask whether the limiting {\it statistical distribution} of the points $R^{-n}(z)$ can be computed. As we
noted above, for all $z\in\riem$ except at most two, and every continuous test function $\psi$, the averages
$$\frac{1}{d^n}\sum_{w\in R^{-n}(z)} \psi(w)\underset{n\to\infty}{\longrightarrow}\int \psi d\lambda,$$
where $\lambda$ is the Brolin-Lyubich probability measure \cite{Brolin,Lyubich} with $\text{supp}(\lambda)=J(R)$.
We can thus ask whether the value of the integral on the right-hand side can be algorithmically computed with an arbitrary precision.
Even if 
$J(R)=\supp(\lambda)$ is not a computable set,  the answer does not {\it a priori} have to be negative. Informally speaking,
a positive answer would imply a dramatic
difference between the rates of convergence in the following two limits:
$$\lim R^{-n}(z)\underset{\text{\small Hausdorff}}{\longrightarrow} J(R)\text{ and }\lim \frac{1}{d^{n}}\sum_{w\in R^{-n}(z)} \delta_w
\underset{\text{\small weak}}{\longrightarrow} \lambda.$$

Indeed, as was shown in \cite{BBRY}:

\begin{thm}
  \label{BL-computable}
 The Brolin-Lyubich measure of $R$ is always computable by a TM with an oracle for the coefficients of $R$.
\end{thm}

\noindent
Even more surprisingly, the result of Theorem~\ref{BL-computable} is {\it uniform}, in the sense that there is a single algorithm that takes the rational map $R$ as a parameter
and computes the corresponding Brolin-Lyubich measure.
Using the analytic tools given by the work of Dinh and Sibony \cite{DinhSib}, the authors of \cite{BBRY} also got the following complexity bound:

\begin{thm}
\label{BL-comp-2}
  For each rational map $R$, there is an algorithm $\cA(R)$
  that computes the Brolin-Lyubich measure in exponential time.
  \end{thm}

\noindent
The running time of $\cA(R)$ will be of the form $\exp (c(R)\cdot n)$, where $n$ is 
the precision parameter, and $c(R)$ is a constant that depends only on the map $R$ (but
not on $n$). Theorems~\ref{BL-computable} and \ref{BL-comp-2}  are not directly comparable, since the latter bounds the growth
of the computation's running time in terms of the precision parameter, while  the former
gives a single algorithm that works for all rational functions $R$.

As was pointed out above, 
the Brolin-Lyubich measure for a polynomial coincides with the harmonic measure of the complement of the filled Julia set. In view of Theorem~\ref{comput-filled},
it is natural to ask what property of a computable 
compact set in the plane ensures computability of the harmonic measure of the complement. We recall that a compact set $K\subset\riem$ which contains at least two points is {\it uniformly perfect} if the moduli of the ring domains 
separating $K$ are bounded from above. Equivalently, there exists some $C>0$ such that for any 
$x\in K$ and $r>0$, we have 
$$\left(B(x,Cr)\setminus B(x,r)\right)\cap K=\emptyset\implies K\subset B(x,r).$$ In particular, every connected set is uniformly perfect. 
As was shown in \cite{BBRY}:

\begin{thm}
  \label{unif-perf}
 If a closed set $K\subset \CC$ is computable and uniformly perfect, and has a connected complement, then the harmonic measure of the complement is 
computable.
\end{thm}
It is well-known \cite{MR92} that filled Julia sets are uniformly perfect. Theorem~\ref{unif-perf} thus implies Theorem~\ref{BL-computable} in the polynomial case.
Computability of the set $K$ is not enough to ensure computability of the harmonic measure:
in \cite{BBRY} the authors presented a counter-example of a computable closed
set with a non-computable harmonic measure of the complement.

\subsection{Applications of computable Carath{\'e}odory Theory to Julia sets: External rays and their impressions}

Informally (see \cite{BY-book} and \cite{BBRY} for a more detailed discussion),
the parts of the Julia set $J_c$ which are hard to compute are ``inward pointing'' decorations, forming narrow fjords of $K_c$.
If the fjords are narrow enough, they will not appear in a finite-resolution image of $K_c$, which explains how the former can
be computable even when $J_c$ is not. Furthermore, a very small portion of the harmonic measure resides in the fjords, again
explaining why it is always possible to compute the harmonic measure.

Suppose the Julia set $J_c$ is connected, and denote 
$$\phi_c:\hat\CC\setminus \overline{\DD}\to\hat\CC\setminus K_c$$
 the unique conformal mapping satisfying the normalization $\phi_c(\infty)=\infty$ and $\phi_c'(\infty)=1$.
Carath{\'e}odory Theory (see e.g. \cite{Mil} for an exposition) implies that $\phi_c$ extends continuously to map 
the unit circle $\uc$ onto the {\it Carath{\'e}odory completion} $\hat J_c$ of the Julia set. 
An element of the set $\hat J_c$ is a {\it prime end} $p$ of $\CC\setminus K_c$. The impression $\imp(p)$ of a prime end 
is a subset of $J_c$ which should roughly be thought as a part of $K_c$ accessible by a particular approach from the exterior.
The harmonic measure $\omega_c$ can be viewed as the pushforward of the Lebesgue measure on $\uc$ onto the set of prime end
impressions.

In view of the above quoted results, from the point of view of computability, prime end impressions should be seen as borderline
objects. On the one hand, they are subsets of the Julia set, which may be non-computable, on the other they are ``visible from
infinity'', and as we have seen accessibility from infinity generally implies computability.

It is thus natural to ask whether 
 the impression of a prime end of $\hat\CC\setminus K_c$ always computable by a TM with an oracle for $c$.
To formalize the above question, we need to describe a way of specifying a prime end. We recall that {\it the external ray}
$R_\alpha$ of angle $\alpha\in \RR/\ZZ$ is the image under $\phi_c$ of the radial line $\{re^{2\pi i\alpha}\;:\;r>1\}$.
The curve $$R_\alpha=\phi_c(\{re^{2\pi i\alpha}\;:\;r>1\})$$
lies in $\hat\CC\setminus K_c$. The {\it principal impression of an external ray} $\pimp(R_\alpha)$ is the set of limit points of $\phi_c(re^{2\pi i\alpha})$ as $r\to 1$.
If the principal impression of $R_\alpha$ is a single point $z$, we say that $R_\alpha$ {\it lands} at $z$. External rays play a very important role in the study of polynomial dynamics.
 
It is evident that every principal impression is contained in the impression of a unique prime end. We call the impression of this prime end the 
{\it prime end impression of an external ray} and denote it
$\imp(R_\alpha)$. A natural refinement of the first question is the following:
 suppose $\alpha$ is a computable angle;  
is the prime end impression $\imp(R_\alpha)$ computable?
As was shown in \cite{binder2015non}, the answer is emphatically negative:

\begin{thm}
  There exists a computable complex parameter $c$ and a computable Cantor set of angles 
$C\subset \uc$ such that for every angle $\alpha \in C$, the impression $\imp(R_{\alpha})\subset J_c$ is not computable.
Moreover, any compact subset $K\Subset J_c$ which contains $\imp(R_\alpha)$ is non-computable.
\end{thm}

This statement illustrates yet again how subtle, and frequently counter-intuitive, the answers to natural computability questions may be when it comes to Julia sets, and, by extensions, to other fractal invariant sets in low-dimensional dynamics.

\subsection{On the computability of the Mandelbrot set}

Let us recall that the Mandelbrot set $\M$ is defined to be the \emph{connectedness locus} of the family $f_{c}(z)=z^{2}+c, \,\, c\in \CC$:
the set of complex parameters $c$ for which the Julia set $J(f_{c})$ is connected. The boundary of $\M$ corresponds to the parameters near which the geometry of the Julia set undergoes a dramatic change. For this reason, $\partial \M$ is referred to as the \emph{bifurcation locus}.  As already discussed in Section \ref{juliasets}, $J(f_{c})$ is connected precisely when the critical point $0$ does not escape to infinity. Therefore, $\M$ can be  equivalently defined as
$$
\M = \{ c \in \CC :  \sup_{n\in\NN}|f^{n}(0)| < \infty \}.
$$
$\M$ is widely known for the spectacular beauty of its fractal structure, and an enormous amount of effort has been made in order to understand its topological and geometrical properties.   It is easy to see that $\M$ is a compact set, equal to the closure of its interior, and contained in the disk of radius $2$ centered at the origin. Douady and Hubbard have shown \cite{douady1982iteration} that $\M$ is connected and simply connected.  In this section we will discuss the computability properties of $\M$, a question raised by Penrose in \cite{penrose1991emperor}.  Hertling showed in \cite{HertMandel}:

\begin{thm} The complement  of the Mandelbrot set is a lower-computable open set, and its boundary $\partial \M$ is a lower computable closed set. 
\end{thm}
\begin{proof}
  Note that $$\CC\setminus \M = \bigcup_{n\in\NN} \{c \in \CC :  |f_{c}^{n}(0)|>2 \}.$$  Since $f_{c}^{n}(0)$ is computable as a function of $c$, uniformly in $n$, it follows that the open sets $\{c\in\CC:|f_{c}^{n}(0)|>2\}$ are uniformly lower-computable, which proves the first claim.  For the second claim, a simple way to compute a dense sequence of points in $\partial \M$ is by computing the so-called {\it Misiurewicz parameters},
  for which the critical point of $f_{c}$ is pre-periodic. In other words,  the union of sets $$W_{l,p}=\{c\;|\;f_{c}^{l+p}=f_{c}^{l}\}\text{ for }l,p\in\NN$$
  is dense in $\partial \M$.
\end{proof}

In virtue of Theorem \ref{comp-riem-map}, we note the following:
\begin{corollary}\label{RiemannMandel} The Riemann map $$\phi:\hat{\CC}\setminus \overline{\DD} \to \hat{\CC} \setminus \M$$ sending the complement of the unit disk to the complement of the Mandelbrot set, is computable. 
\end{corollary}

It is unknown whether the whole of $\M$ is a lower-computable set. A positive answer would imply computability of $\M$.
In fact, Hertling \cite{HertMandel} also showed:
\begin{thm}
Density of Hyperbolicity Conjecture implies that $\M$ is lower-computable, and hence, computable.
\end{thm}
Indeed, let $U_m$ be defined as the set of parameters $c\in\CC$ such that $f_c$ has a point $p\in\CC$ with $f_c^m(p)=p$ and $|Df_c^m(p)|<1$. It is not hard to see that such a set is lower-computable. Density of Hyperbolicity implies that the sets $U_m$ are dense in $\cM$.

Note that the same statements hold for the $n-1$ complex dimensional connectedness locus of the family of polynomials of degree $n$, with similar proofs. On the other hand, it is possible to construct a one-parameter complex family of polynomials in which the corresponding objects are not computable. In
\cite{Mandel}, Coronel, Rojas, and Yampolsky have recently shown:

\begin{thm} There exists an explicitly computable complex number $\lambda$ such that the bifurcation locus  of the one parameter family $$ f_{c}(z) = \lambda z + cz^{2} + z^{3} $$ is not computable.  
\end{thm}

An interesting related question is the computability of the area of $\M$. Since $\CC\setminus \M$
is a lower-computable set, it follows that the area of $\M$ is an upper-computable real number.  The following simple conditional implication holds:

\begin{proposition}If the area of $\M$ is computable, then $\M$ is computable. 
\end{proposition}
\begin{proof} Since $\CC\setminus \M$ is a lower-computable set, if the area of $\M$ was a computable real number, by Proposition 2.3.1.1 from \cite{Cristobal} Lebesgue measure restricted  to $\M$ would be a computable measure. By Proposition~\ref{comp-supp}, the support of this measure is  lower-computable.  Since the support in this case is $\M$, the result follows. 
\end{proof}

Let us also note another famous conjecture about the Mandelbrot set:

\medskip
\noindent
{\bf Conjecture (MLC).} {\it The Mandelbrot set $\M$  is locally connected.}

\medskip
MLC is known (see \cite{Do2}) to be stronger than Density of Hyperbolicity Conjecture, and thus also implies computability of $\M$ and of $\partial \M$. In this case, by virtue of Corollary~\ref{RiemannMandel} and Theorem \ref{main3}, it would be natural to ask whether $\M$ admits a computable Carath\'eodory modulus. Some partial results in this direction are known (see e.g. \cite{HubbardYoc}). In all of them, local connectedness is established at subsets of points of $\cM$ by providing a constructive Carath{\'e}odory modulus at these points.

\bibliographystyle{plain}
\bibliography{Bib_RY}

\end{document}